\documentclass[reqno,b5paper]{amsart}

\usepackage{amsmath}
\usepackage{amsthm}
\usepackage{amssymb}
\usepackage{enumerate}
\usepackage[mathscr]{eucal}
\usepackage{eqlist}

\theoremstyle{plain}
\newtheorem{Thm}{Theorem}[section]
\newtheorem{Prop}[Thm]{Proposition}
\newtheorem{Lem}{Lemma}[section]
\newtheorem{Cor}{Corollary}[Thm]
\theoremstyle{definition}
\newtheorem{Def}{Definition}[section]
\newtheorem{Rem}{\textnormal{\textbf{Remark}}}
\newtheorem*{acknowledgement}{\textnormal{\textbf{Acknowledgement}}}
\theoremstyle{remark}
\newtheorem{Exm}{Example}

\numberwithin{equation}{subsection}

\begin{document}

\title[]{Note on the characteristic rank of vector bundles}

\author{Aniruddha C. Naolekar \and Ajay Singh Thakur}

\newcommand{\acr}{\newline\indent}

\address{Stat-Math UNit \acr Indian Statistical Institute \acr 8th Mile, Mysore Road, RVCE Post \acr Bangalore 560059 \acr INDIA.}

\email{ani@isibang.ac.in, thakur@isibang.ac.in}



\begin{abstract}
We define the notion of characteristic rank, $\mathrm{charrank}_X(\xi)$, of a real vector bundle $\xi$ over a connected finite $CW$-complex $X$. This is a bundle-dependent version of the notion of characteristic rank introduced by J\'{u}lius Korba\v{s} in 2010. We obtain bounds for the cup length of manifolds in terms of the characteristic rank of vector bundles generalizing a theorem of Korba\v{s} and compute the characteristic rank of vector bundles over the Dold manifolds, the Moore spaces and the stunted projective spaces amongst others. 
\end{abstract}

\maketitle

\section{Introduction}

Recently, J. Korba\v{s} \cite{korbas} has introduced a new homotopy invariant, called the characteristic rank, of a  connected closed smooth manifold $X$. The characteristic rank of a connected closed smooth $d$-manifold $X$, denoted by $\mathrm{charrank}(X)$, is the largest integer $k$, $0\leq k \leq d$, such that every cohomology class $x\in H^j(X;\mathbb Z_2)$, $0\leq j\leq k$ is a polynomial in the Stiefel-Whitney classes of (the tangent bundle of) $X$. 

Apart from being an interesting question in its own right, part of the motivation for computing the characteristic rank comes from a theorem of Korba\v{s} (\cite{korbas}, Theorem 1.1), 
where the author has described a bound for the $\mathbb Z_2$-cup-length of (unorientedly) null cobordant closed smooth manifolds in terms of their charateristic rank. More specifically, Korba\v{s} has proved the following. 

\begin{Thm}\label{korbastheorem} {\em (\cite{korbas}, Theorem\,1.1)}
Let $X$ be a closed smooth connected $d$-dimensional manifold unorientedly
cobordant to zero. Let $\widetilde{H}^r(X;\mathbb Z_2)$, $r < d$, be the first nonzero reduced cohomology group
of $X$. Let $z$ ($z < d -1$) be an integer such that for $j\leq z$ each element of $H^j(X;\mathbb Z_2)$ can
be expressed as a polynomial in the Stiefel-Whitney classes of the manifold $X$. Then we have that 
$$\mathrm{cup}(X)\leq 1+\frac{d-z-1}{r}.$$
\end{Thm}

Recall that the $\mathbb Z_2$-cup-length, denoted by $\mathrm{cup}(X)$, of a space $X$ is the largest integer $t$ such that there exist classes $x_i\in H^*(X;\mathbb Z_2)$, $\mathrm{deg}(x_i)\geq 1$, such that the cup product $x_1\cdot x_2\cdots x_t\neq 0$. We mention in passing that the $\mathbb Z_2$-cup-length is well known to have connections with the Lyusternik-Shnirel'man category of the space. 

With the computation of the characteristic rank in mind,   
Balko and Korba\v{s} \cite{balkokorbas} obtained bounds for the characteristic rank of manifolds $X$ which occur as total spaces of smooth fiber bundles with fibers totally non-homologous to zero, and also in the situation where, additionally, $X$ itself is null cobordant (see \cite{balkokorbas}, Theorems 2.1 and 2.2). 

It is useful to think of the characteristic rank of a manifold as the characteristic rank ``with respect to the tangent bundle'' and introduce bundle-dependency 
as in the definition below. 

\begin{Def}\label{def}
Let $X$ be a connected, finite $CW$-complex and $\xi$ a real vector bundle over $X$. The characteristic rank of the vector bundle $\xi$ over $X$, denoted by $\mathrm{charrank}_X(\xi)$, is by definition 
the largest integer $k$, $0\leq k\leq \mathrm{dim}(X)$, such that every cohomology class $x\in H^j(X;\mathbb Z_2)$, $0\leq j\leq k$, is a polynomial in the  Stiefel-Whitney classes $w_i(\xi)$ of $\xi$. The upper characteristic rank of $X$, denoted by $\mathrm{ucharrank}(X)$, is the maximum of $\mathrm{charrank}_X(\xi)$ as $\xi$ varies over all 
vector bundles over $X$. 
\end{Def}

Thus, if $X$ is a connected closed smooth manifold, then $\mathrm{charrank}_X(TX)=\mathrm{charrank}(X)$ where $TX$ is the tangent bundle of $X$. Note that if $X$ and $Y$ are homotopically equivalent closed connected smooth manifolds, then $\mathrm{ucharrank}(X)=\mathrm{ucharrank}(Y)$.

In this note we discuss some general properties of $\mathrm{charrank}(\xi)$ and give a complete description of $\mathrm{charrank}_X(\xi)$ of vector bundles $\xi$ over $X$ when $X$ is: a product of spheres, the real  and complex projective spaces, the Dold manifold $P(m,n)$, the Moore space $M(\mathbb Z_2,n)$ and the stunted projective spaces $\mathbb R\mathbb P^n/\mathbb R\mathbb P^m$. We now briefly describe the contents of this note. 

For a connected finite $CW$-complex $X$, let $r_X$ denote the smallest positive integer such that $\widetilde{H}^{r_X}(X;\mathbb Z_2)\neq 0$. 
In the case that such an integer does not exist, that is, all the reduced cohomology groups $\widetilde{H}^i(X;\mathbb Z_2)=0$, $1\leq i\leq \dim (X)$, we set $r_X=\dim (X)+1$. In any case, $r_X\geq 1$. 

Making the definition of the characteristic rank bundle-dependent gives the following theorem which is a straighforward generalisation of Theorem\,\ref{korbastheorem}. In this form the theorem yields sharper bounds on the cup-length in certain cases (see Examples\,\ref{4.6} and \ref{4.7} below). We shall prove the following. 

\begin{Thm}\label{zeroththeorem}
Let $X$ be a connected closed smooth $d$-manifold. Let $\xi$ be a vector bundle over $X$ satisfying the following:
\begin{itemize}
\item there exists $k$, $k\leq \mathrm{charrank}_X(\xi)$, such that every monomial $$w_{i_1}(\xi)\cdots w_{i_r}(\xi), 0\leq i_t\leq k,$$ of total degree $d$ is zero.
\end{itemize}
Then, 
$$\mathrm{cup}(X)\leq 1+\frac{d-k-1}{r_X}.$$
\end{Thm}

We note that if $X$ is an unoriented boundary, then $\xi=TX$ satisfies the conditions of the theorem above  with $k=\mathrm{charrank}_X(TX)$. In this theorem we do not assume that $X$ is an unoriented boundary.

If $X$ is an unoriented boundary and there exists a vector bundle $\xi$ over $X$ with $k$ satisfying the conditions of the above theorem, such that 
\begin{equation}\label{equation}
\mathrm{charrank}(X)=\mathrm{charrank}_X(TX)<k\leq \mathrm{charrank}_X(\xi),\end{equation}
then the bound for $\mathrm{cup}(X)$ using $k$ is sharper than that obtained from Theorem\,\ref{korbastheorem}. We note that over the null cobordant manifold $S^d\times S^m$, $d=2,4,8$, and $m\neq 2,4,8$, there exists a vector bundle $\xi$ and an integer $k$ satisfying the conditions of Theorem\,\ref{zeroththeorem} and equation \ref{equation} (see Examples\,\ref{4.6}, \ref{4.7} below). 

If $X$ is a connected closed smooth manifold with $\mathrm{ucharrank}(X)=\mathrm{dim}(X)$, it turns out that the cup-length $\mathrm{cup}(X)$ of $X$ can be computed as the maximal length 
of a non-zero product of the Stiefel-Whitney classes of a suitable bundle over $X$. We prove the following.

\begin{Thm}\label{cuplength}
Let $X$ be a connected closed smooth $d$-manifold. If $$\mathrm{ucharrank}(X)=\mathrm{dim}(X),$$ then there exists a vector bundle $\xi$ over $X$ such that 
$$\mathrm{cup}(X)=\max\{k\mid \mbox{there exist $ i_1,\ldots , i_k\geq 1$ with $w_{i_1}(\xi)\cdots w_{i_k}(\xi)\neq 0$}\}.$$
\end{Thm}

Making the definition of characteristic rank bundle-dependent allows us, under certain conditions, to construct an epimorphism  $\widetilde{KO}(X)\longrightarrow \mathbb Z_2$. It is clear from the definition that $\mathrm{charrank}_X(\xi)=\mathrm{charrank}_X(\eta)$ if $\xi$ and $\eta$ are (stably) isomorphic. Let $\mathrm{Vect}_{\mathbb R}(X)$ denote the semi-ring of isomorphism classes of real vector bundles over $X$. We then have a function 
$$f:\mathrm{Vect}_{\mathbb R}(X)\longrightarrow \mathbb Z_2$$
defined by $f(\xi)=\mathrm{charrank}_X(\xi) ~~~~~(\mbox{mod}~~~~2)$. We observe that under certain restrictions on the values of $\mathrm{charrank}_X(\xi)$ the function $f$ is actually a semi-group homomorphism. More precisely we prove the following.  

\begin{Thm}\label{maintheorem}
Let $X$ be a connected finite $CW$-complex with $r_X=1$. Assume that for any vector bundle $\xi$ over $X$, $\mathrm{charrank}_X(\xi)$ is either $r_X-1=0$ or  an odd integer. 
Assume that $\mathrm{ucharrank}(X)\geq 1$. Then the function 
$$f:\mathrm{Vect}_{\mathbb R}(X)\longrightarrow \mathbb Z_2$$ 
defined by $f(\xi)=\mathrm{charrank}_X(\xi)~~~~~\mbox{\em (mod $2$)}$ is a surjective semi-group homomorphism and hence gives rise to a surjective group homomorphism 
$\widetilde{f}:KO(X)\longrightarrow \mathbb Z_2$. Furthermore, this restricts to an epimorphism $\tilde{f}:\widetilde{KO}(X)\longrightarrow \mathbb Z_2$. 
\end{Thm} 

 The function $f$ defined in the theorem above is 
in general not a semi-ring homomorphism (see Remark\,\ref{importantremark}). 
There is a large class of spaces that satisfy the conditions of this theorem. We prove the following.

\begin{Thm}\label{mainproposition} 
\begin{enumerate}
\item Let $X = \mathbb R\mathbb P^n$. 
Then $\mathrm{ucharrank}(X)= n$ and for any vector bundle $\xi$ over $X$, the characteristic rank $\mathrm{charrank}_X(\xi)$ is either $r_X-1=0$ or is $n$.  
\item Let $X=S^1\times \mathbb C\mathbb P^n$. Then $\mathrm{ucharrank}(X)= 2n+1$ and for any vector bundle $\xi$ over $X$, the characteristic rank $\mathrm{charrank}_X(\xi)$ either is $r_X-1=0$, $1$ or $2n+1$.  
\item Let $X$ be the Dold manifold $P(m,n)$. Then $\mathrm{ucharrank}(X)= 2n +m$ and for any vector bundle $\xi$ over $X$, the characteristic rank $\mathrm{charrank}_X(\xi)$ is either $r_X-1=0$, $1$ or $2n+m$.   
\end{enumerate}
\end{Thm}

Recall that the Dold manifold $P(m,n)$ is the quotient of $S^m\times \mathbb C\mathbb P^n$ by the fixed point free involution $(x,z)\mapsto (-x,\bar{z})$.

In this note we concentrate on the computational part of characteristic rank of vector bundles. We compute the characteristic rank of vector bundles over products of spheres $S^d\times S^m$, the real and complex projective spaces, the spaces $S^1\times \mathbb C\mathbb P^n$, the Dold manifold $P(m,n)$, the 
Moore space $M(\mathbb Z_2,n)$ and the stunted projective space $\mathbb R\mathbb P^n/\mathbb R\mathbb P^m$. We also prove some general facts about characteristic rank of vector bundles.

The paper is organized as follows. In Section 2 we prove some general facts about $\mathrm{charrank}(\xi)$. In Section 3 we prove Theorems\,\ref{zeroththeorem}, \ref{cuplength} and \ref{maintheorem}. Finally, in Section 4, we compute $\mathrm{charrank}_X(\xi)$ where $X$ is one of the following spaces: the product of spheres $S^d\times S^m$, the real and complex projective spaces, the product $S^1\times \mathbb C\mathbb P^n$, the Dold manifold $P(m,n)$, the Moore space $M(\mathbb Z_2,n)$  and the stunted projective space.  

{\bf Convention.} By a space we shall mean a connected finite $CW$-complex. All vector bundles are real unless otherwise stated.


\section{Generalities}


In this section we make some general observations about $\mathrm{charrank}(\xi)$. 
Recall that, for a space $X$, $r_X$ denotes the smallest positive integer for which the reduced cohomology group $\widetilde{H}^{r_X}(X;\mathbb Z_2)\neq 0$, and if such an $r_X$ does not exist, then we set $r_X=\mathrm{dim}(X)+1$. Then for any vector bundle $\xi$ over $X$ we have  
$$r_X-1\leq \mathrm{charrank}_X(\xi)\leq \mathrm{ucharrank}(X).$$
We begin with some easy observations. 

\begin{Lem} \label{firstlemma} Let $\xi$ and $\eta$ be any two vector bundles over a space $X$.     
\begin{enumerate}
\item If $w_{r_X}(\xi)=0$, then $\mathrm{charrank}_X(\xi)=r_X-1$;
\item If $w(\xi)=1$, then $\mathrm{charrank}_X(\xi)=r_X-1$. 
\item If $w(\eta)=1$, then $\mathrm{charrank}_{X}(\xi\oplus\eta)=\mathrm{charrank}_X(\xi)$. Hence if $\widetilde{KO}(X)=0$, then $\mathrm{charrank}_X(\xi)=r_X-1$ for any vector bundle over $X$;
\item If $\xi$ and $\eta$ are stably isomorphic, then $\mathrm{charrank}_X(\xi)=\mathrm{charrank}_X(\eta)$; 
\item There exists a vector bundle $\theta$ over $X$ such that $\mathrm{charrank}_{X}(\xi\oplus \theta)=r_X-1$.
\end{enumerate}
\end{Lem} 
\begin{proof}(1) follows from the definition. Clearly, (2) follows from (1). 
To prove (3) we note that since $w(\xi\oplus\eta)=w(\xi)$, we have $\mathrm{charrank}_{X}(\xi\oplus\eta)=\mathrm{charrank}_X(\xi)$. As $\widetilde{KO}(X)=0$, we have $\xi\oplus\varepsilon\cong\varepsilon'$. Hence 
$$\mathrm{charrank}_X(\xi)=\mathrm{charrank}_X(\xi\oplus\varepsilon)=\mathrm{charrank}_X(\varepsilon')=r_X-1.$$
This completes the proof of (3). 
Next, if $\xi$ and $\eta$ are stably isomorphic, we have   
 $\xi\oplus \varepsilon\cong \eta\oplus\varepsilon'$ where $\varepsilon$ and $\varepsilon'$ are trivial vector bundles. Hence (4) follows from (3). 
Finally, as $X$ is compact, given $\xi$ we can find a vector bundle $\theta$ such that $\xi\oplus\theta\cong \varepsilon$. Hence (5) follows from (4) and (2).
\end{proof}
\begin{Lem} \label{secondlemma}
Let $X$ be a space and $1 \leq r_X \leq \dim(X)$
\begin{enumerate}
\item If $\mathrm{ucharrank}(X)\geq r_X$, then $\mathrm{dim}_{\mathbb Z_2}H^{r_X}(X;\mathbb Z_2)= 1$. 
\item If $r_X$ is not a power of $2$, then $\mathrm{ucharrank}(X)=r_X-1$. 
\end{enumerate}
\end{Lem}
\begin{proof} If $\xi$ is a vector bundle over $X$ with $\mathrm{charrank}_X(\xi)\geq r_X$, then by Lemma\,\ref{firstlemma} (1), $w_{r_X}(\xi)\neq 0$. This 
forces the equality $\mathrm{dim}_{\mathbb Z_2}H^{r_X}(X;\mathbb Z_2)= 1$ and proves (1). It is known that for any vector bundle $\xi$, the smallest integer 
$k$ such that $w_k(\xi)\neq 0$ is always a power of $2$ (see, for example, \cite{ms}, page 94). Lemma\,\ref{firstlemma} (1) now completes the proof of (2). \end{proof}

Let $Y$ be a space and and let $X=\Sigma Y$ be the suspension of $Y$. Then any cup-product of elements of positive degree in $H^*(X;\mathbb Z_2)$ is zero. 
The following lemma is an easy consequence of this fact and we omit the proof. 

\begin{Lem}\label{thirdlemma}
Let $Y$ be a space and $X=\Sigma Y$. Let $k_X$ be an integer defined by 
$$k_X=\mathrm{max}\{k\mid \mathrm{dim}_{\mathbb Z_2}H^j(X;\mathbb Z_2)\leq 1, 0\leq j\leq k, k\leq \mathrm{dim}(X)\}.$$
Let $\xi$ be any vector bundle over $X$. Then, $\mathrm{charrank}_X(\xi)\leq k_X$. In particular, $\mathrm{ucharrank}(X)\leq k_X$. \qed
\end{Lem}

\begin{Lem}\label{inequalitylemma}
Let $f:X\longrightarrow Y$ be a map between spaces. If $f^*:H^*(Y;\mathbb Z_2)\longrightarrow H^*(X;\mathbb Z_2)$ is surjective, then 
$$\mathrm{charrank}_{X}(f^*\xi)\geq \min\{\mathrm{charrank}_{Y}(\xi), \mathrm{dim}(X)\}$$
for any vector bundle $\xi$ over $Y$. 
\end{Lem}
\begin{proof} As $w_i(f^*\xi)=f^*(w_i(\xi))$, the surjectivity of $f^*$ implies that every cohomology class in $H^*(X;\mathbb Z_2)$ 
of degree at most $\mathrm{charrank}_{Y}(\xi)$ is a polynomial in the Stiefel-Whitney classes of $f^*\xi$. 
If $\mathrm{charrank}_{Y}(\xi)\geq \mathrm{dim}(X)$, then  $$\mathrm{charrank}_{X}(f^*\xi)=\mathrm{dim}(X).$$ If $\mathrm{charrank}_{Y}(\xi)\leq \mathrm{dim}(X)$, then $\mathrm{charrank}_{Y}(\xi)\leq \mathrm{charrank}_{X}(f^*\xi)\leq \mathrm{dim}(X)$. \end{proof}

Before mentioning further general properties of the characteristic rank we record the characteristic rank of vector bundles over the sphere. The description of 
the characteristic rank of vector bundles over the spheres is an easy consequence of the following theorem due to Atiyah-Hirzebruch (\cite{atiyah}, Theorem\,1), (see also \cite{milnor}). 

\begin{Thm}\label{atiyah}{\em (\cite{atiyah}, Theorem\,1)} There exists a real vector bundle $\xi$ over the sphere $S^d$ with $w_d(\xi)\neq 0$ only for 
$d=1,2,4$, or $8$.\qed
\end{Thm}

For the Hopf bundle $\nu_d$ over $S^d$ ($d=1,2,4,8$), the Stiefel-Whitney class $w_d(\nu_d)$ is not zero. Thus, 
$$\mathrm{ucharrank}(S^d)=\left\{\begin{array}{cl}
d & \mbox{if $d=1,2,4$, or $8$}\\
d-1 & \mbox{otherwise.}\end{array}\right.$$
Note that $\mathrm{charrank}(S^d)=d-1$. We shall use the above description of characteristic rank of vector bundles over the spheres in the sequel without explicit reference. 

Suppose that $\pi:S^d\longrightarrow X$ is a $k$-sheeted covering with $k>1$ odd. Since $X\cong S^d/G$, where $G$ is a finite group with $|G|=k$, we have that $d$ is odd. By Proposition\,3G.1 of \cite{hatcher}, the homomorphism $\pi^*:H^i(X;\mathbb Z_2)\longrightarrow H^i(S^d;\mathbb Z_2)$ is a monomorphism with image the $G$-invariant elements for all $i\geq 0$. In particular, $H^i(X;\mathbb Z_2)=0$, $0<i<d$ and $\pi^*:H^d(X;\mathbb Z_2)\longrightarrow H^d(S^d;\mathbb Z_2)\cong\mathbb Z_2$ is an isomorphism. 
 Thus we have the following corollary to Theorem \ref{atiyah}.

\begin{Cor}\label{lens1}
Assume that $\pi:S^d\longrightarrow X$ is a $k$-sheeted covering with an odd $k>1$ and $d\neq 1$. Then $w(\xi)=1$ for any vector bundle $\xi$ over $X$ and we have  $\mathrm{ucharrank}(X)=d-1$. 

\end{Cor}
\begin{proof} If $0 < i < d$, then obviously $w_i(\xi)= 0$. In addition, for any $\xi$ we have now  $\pi^*(w_d(\xi)) = w_d(\pi^*\xi) = 0$ by Theorem \ref{atiyah}. Since, $\pi^*$ is injective, we thus have $w_d(\xi)= 0$. We know that $H^d(X,\mathbb Z_2) \cong \mathbb Z_2$; this implies that $\mathrm{charrank}_X(\xi)\leq d-1$ for any $\xi$. The inequality $\mathrm{charrank}_X(\xi)\geq d-1$ for any $\xi$ is clear.
\end{proof}
\begin{Exm}
Let $L=L_m(\ell_1,\ldots ,\ell_n)$ denote the lens space which is a quotient of $S^{2n-1}$ by a free action of the cyclic group $\mathbb Z_m$ (see \cite{hatcher}, page 144). Then, we have an $m$-sheeted covering $\pi: S^{2n-1}\longrightarrow L$. If $n >1$ and $m$ is odd, then for any vector bundle $\xi$ over $L$, the total Stiefel-Whitney class $w(\xi)=1$. In particular, $\mathrm{ucharrank}(L)= 2n-2$.

\end{Exm}

There are conditions under which one can obtain a natural upper bound on the upper characteristic rank of a space. 
One such condition is the existence of a spherical class. Recall that a cohomology class $x\in H^k(X;\mathbb Z_2)$ is spherical if there exists a map $f:S^k\longrightarrow X$ with $f^*(x)\neq 0$. Note that a spherical class $x\in H^k(X;\mathbb Z_2)$ is indecomposable as an element of the cohomology ring. 
We shall show that the upper characteristic rank of a space is bounded above by the degree of a spherical class in most cases.  

\begin{Prop}\label{covering}
Let $X$ be a space and assume that $x\in H^k(X;\mathbb Z_2)$ is spherical, $k\neq 1,2,4,8$. Then there does not exist a vector bundle $\xi$ over $X$ with $w_k(\xi)=x$ and we have $\mathrm{charrank}_X(\xi) < k$ for any $\xi$. As a consequence,  for any covering $\pi:E\longrightarrow X$, we have 
$\mathrm{ucharrank}(E)<k$ (in particular, $\mathrm{ucharrank}(X)<k$ ). 

\end{Prop}
\begin{proof} Assume that $\xi$ is a vector bundle over $X$ with $w_k(\xi)= x\neq 0$. Let $f:S^k\longrightarrow X$ be a map with $f^*(x)\neq 0$. Then one has  $w_k(f^*\xi)=f^*(w_k(\xi))\neq 0$, which is impossible by Theorem \ref{atiyah}. Hence there is no such $\xi$. Now since there is no $\xi$ with $w_k(\xi) = x$, and $x$ is indecomposable, we see that $\mathrm{charrank}_X(\xi) < k$ for any $\xi$. The rest of the claim follows from the fact that $f$ factors through the covering projection $\pi: E \longrightarrow X$. Indeed, we have $f = \pi \circ g$ for some $g : S^k \longrightarrow E$, and then $g^*(\pi^*(x))=f^*(x)\neq 0$, which means that the class $\pi^*(x)$ is spherical. The proof is finished by taking $E$ in the role of $X$ in the preceding considerations. \end{proof}

When a spherical class has degree $k=1,2,4$, or $8$, there can exist vector bundles of characteristic rank greater than or equal to the degree of the spherical class. For example, the sphere $S^k$ with $k=1,2,4$, or $8$ 
has upper characteristic rank equal to $k$. The complex projective space $\mathbb C\mathbb P^n$ has a spherical class in degree $2$, however 
$\mathrm{ucharrank}(\mathbb C\mathbb P^n)=2n$ (see Example\,\ref{complexprojective}). When a spherical class exists in degree $1,2,4$ or $8$, we have the following observation:  

{\bf Observation:} Let $X$ be a space and assume that $x\in H^k(X;\mathbb Z_2)$ is spherical, where $k=1,2,4,8$. Let $f:S^k\longrightarrow X$ be a map with $f^*(x)\neq 0$. Then for a vector bundle $\xi$ over $X$ with $\mathrm{charrank}_X(\xi)\geq k$, we can express $x$ as a polynomial $P(w_1(\xi),w_2(\xi),\ldots w_k(\xi))$. But then $0 \neq f^*(x) = f^*(P(w_1(\xi),w_2(\xi),\ldots w_k(\xi)))$ = $  P(f^*(w_1(\xi)),f^*(w_2(\xi)),\ldots,f^*(w_k(\xi)))$. Hence $f^*(w_k(\xi)) \neq 0$. Thus for any vector bundle $\xi$ over $X$ with $\mathrm{charrank}_X(\xi)\geq k$, we have $w_k(\xi)\neq 0$. 

When $X$ is a connected closed smooth $d$-manifold, the characteristic rank, $\mathrm{charrank}_X(\xi)$, of $\xi$ takes values in a certain specific range. 
We prove the following. 

\begin{Thm}\label{minusone}
Let $X$ be a connected closed smooth $d$-manifold. Assume that $2r_X\leq d$. Then, for any vector bundle $\xi$ over $X$, $\mathrm{charrank}_X(\xi)$ is either $d$ or less than $d-r_X$. 
\end{Thm}
\begin{proof} Let $\xi$ be a vector bundle over $X$ with $\mathrm{charrank}_X(\xi)\geq d-r_X$. We shall show that $\mathrm{charrank}_X(\xi)=d$. 
Since, by Poincar\'e duality, the groups $H^j(X;\mathbb Z_2)=0$ for $d-r_X<j<d$, the proof will be complete if the non-zero element in $H^d(X;\mathbb Z_2)$ is a polynomial in the Stiefel-Whitney classes of $\xi$. 
As $\mathrm{charrank}_X(\xi)\geq d-r_X\geq r_X$, then by Lemma\,\ref{secondlemma}, $H^{r_X}(X;\mathbb Z_2)\cong \mathbb Z_2$. Hence  
$H^{d-r_X}(X;\mathbb Z_2)\cong \mathbb Z_2$. Let $a,b,x$ denote the non-zero cohomology classes in degrees $r_X$, $d-r_X$ and $d$ respectively. The non-degeneracy of the pairing 
$$H^{r_X}(X;\mathbb Z_2)\otimes H^{d-r_X}(X;\mathbb Z_2)\longrightarrow H^d(X;\mathbb Z_2)$$ 
implies that $a\cdot b=x$. As $\mathrm{charrank}_X(\xi)\geq d-r_X\geq r_X$ we have, by Lemma\,\ref{firstlemma} (1), $w_{r_X}(\xi)\neq 0$ and hence $w_{r_X}(\xi)=a$ and 
$b=p(w_1(\xi), w_2(\xi), \ldots)$ is a polynomial in the Stiefel-Whitney classes of $\xi$. This shows that 
$$x=w_{r_X}(\xi)\cdot p(w_1(\xi),w_2(\xi), \ldots)$$
is a polynomial in the Stiefel-Whitney classes of $\xi$. This completes the proof of the theorem. \end{proof}

Let $X$ be a connected closed smooth $d$-manifold. 
If $X$ is an unoriented boundary, then any monomial in the Stiefel-Whitney classes of $X$ of total degree $d$ is zero (see \cite{ms}, Theorem 4.9). Hence the non-zero element 
in $H^d(X;\mathbb Z_2)$ is never a polynomial in the Stiefel-Whitney classes of $X$. We thus have the following corollary.  

\begin{Cor}\label{unnecessarycorollary}
Let $X$ be a connected closed smooth $d$-manifold. Assume that $2r_X\leq d$. If $X$ is an unoriented boundary, then $\mathrm{charrank}(TX)<d-r_X$. \qed
\end{Cor}

\begin{Rem}\label{korbasremark}
Balko and Korba\v{s} \cite{balko} showed independently the following
stronger version of Corollary \ref{unnecessarycorollary}: For any connected closed smooth $d$-dimensional manifold
$X$ that is an unoriented boundary, if $s$, $s\leq \frac{d}{2}$, is (the biggest) such that
$H^s(X;\mathbb Z_2)\neq 0$, then $\mathrm{charrank}(X)<d-s$.
\end{Rem}


\section{Proof of Theorems\,\ref{zeroththeorem}, \ref{cuplength} and \ref{maintheorem}}


In this section we prove Theorems\,\ref{zeroththeorem}, \ref{cuplength},  and \ref{maintheorem}. The proof of Theorem\,\ref{zeroththeorem} is essentially the same as the 
proof of Theorem\,\ref{korbastheorem}. We reproduce it here for completeness. 

{\bf Proof of Theorem\,\ref{zeroththeorem}} 
Let $x=x_1 \cdot x_2\cdots x_s\neq 0$ be a non-zero product of cohomology classes of positive degree and of maximal length. Then $x\in H^d(X;\mathbb Z_2)$. 
If not, then by Poincar\'e duality one can find some $y$ in complementary dimension such that $x\cdot y\neq 0$ contradicting the maximality of $s$. By rearranging, we write 
$$x= \alpha_1\cdots \alpha_m\cdot\beta_1 \cdots \beta_n$$
where $\mathrm{deg}(\alpha_i)\leq k$ and $\mathrm{deg}(\beta_j)\geq k+1$. We note that $n\neq 0$. For otherwise the product $\alpha=\alpha_1\cdots \alpha_m$ which is now a polynomial in $w_1(\xi),\ldots, w_k(\xi)$, would be a non-zero element of total degree $d$ contradicting the assumption on $\xi$. Therefore, 
if $\beta = \beta_1\cdots \beta_n$, then $\mathrm{deg}(\beta)\geq k+1$. Thus $\mathrm{deg}(\alpha)\leq d- (k+1)$. Thus 
$$\begin{array}{rcl}
\mathrm{cup}(X) & = & m+n\\
                 & &\\
                & \leq & \frac{\mathrm{deg}(\alpha)}{r_X} + \frac{\mathrm{deg}(\beta)}{(k+1)}\\
                & &\\
                & = & \frac{\mathrm{deg}(\alpha)}{r_X} + \frac{(d-\mathrm{deg}(\alpha))}{(k+1))}\\
                & &\\
                & = & \frac{((k+1-r_X)\mathrm{deg}(\alpha)+dr_X)}{ r_X(k+1)}\\
                & & \\
                & \leq & \frac{((k+1-r_X)(d-(k+1))+dr_X)}{r_X(k+1)}\\
                & &\\
                &= & 1+\frac{d-k-1}{r_X}.\end{array}$$
This completes the proof. \qed

{\bf Proof of Theorem\,\ref{cuplength}.} Let $\xi$ be any vector bundle over $X$ with 
$$\mathrm{charrank}_X(\xi)=\mathrm{ucharrank}(X)=\mathrm{dim}(X).$$
Let $\mathrm{cup}(X)=k$. We shall show that some product of the Stiefel-Whitney classes of $\xi$ of length $k$ is non-zero. Let 
$$x=x_1\cdot x_2\cdots x_k\neq 0$$
be a non-zero product of cohomology classes $x_i\in H^*(X;\mathbb Z_2)$ with $\mathrm{deg}(x_i)\geq 1$. As $\mathrm{charrank}_X(\xi)=\mathrm{dim}(X)$, each 
$x_i$ can be written as a sum of monomials in the Stiefel-Whitney classes of $\xi$. Thus $x$ can be written as a sum of monomials in the Stiefel-Whitney classes of $\xi$, each of length at least $k$. Note that the monomials of length greater than $k$ are zero by hypothesis. As $x\neq 0$, it follows that some monomial in the Stiefel-Whitney classes of $\xi$ of length $k$ is non-zero. This completes the proof of the theorem.\qed 

\begin{Rem}
\begin{enumerate}
\item The proof of Theorem\,\ref{cuplength} actually shows that if some  product 
$x=x_1\cdots x_t\neq 0$ with $1\leq \mathrm{deg}(x_i)\leq \ell$, then for any vector bundle $\xi$ over $X$ with $\mathrm{charrank}_X(\xi)\geq\ell$ some product of the Stiefel-Whitney classes of $\xi$ of length greater than or equal to $t$ is non-zero. 
\item The conclusion of Theorem\,\ref{cuplength} is not true if $\mathrm{ucharrank}(X)<\mathrm{dim}(X)$. If $X=S^k$, $k\neq 1,2,4,8$, then 
$\mathrm{ucharrank}(X)=k-1<k$, $\mathrm{cup}(X)=1$ however $w(\xi)=1$ for any vector bundle $\xi$ over $X$. 
\end{enumerate}
\end{Rem}

{\bf Proof of Theorem\,\ref{maintheorem}.} First note that the assumption $\mathrm{ucharrank}(X)\geq 1$ is odd clearly implies that the function 
$$f:\mathrm{Vect}_{\mathbb R}(X)\longrightarrow \mathbb Z_2$$ 
defined by 
$$f(\xi)=\mathrm{charrank}_X(\xi) ~~~~~(\mathrm{mod}~~~~~ 2)$$
is surjective. We shall now check that $f$ is a semi-group homomorphism. 
To see this, let $\xi$ and $\eta$ be two bundles over $X$. We have the following cases. 

If $\xi$ and $\eta$ are both orientable, then so is $\xi\oplus\eta$. Hence $w_1(\xi\oplus\eta)=0$. As $r_X=1$, it follows that 
$\mathrm{charrank}_X(\xi\oplus\eta)=0$. The same argument shows that $\mathrm{charrank}_X(\xi)=0=\mathrm{charrank}_X(\eta)$. Thus in this 
case we have $f(\xi\oplus\eta)=f(\xi)+f(\eta)$. 

Next suppose that both $\xi$ and $\eta$ are non-orientable. Then, on the one hand, $\xi\oplus\eta$ is orientable and hence $f(\xi\oplus\eta)=0$ as $r_X=1$. On the other hand, as $\xi$ and $\eta$ are non-orientable, we have
$$f(\xi)= 1 =f(\eta).$$ 
Thus, we have the equality $f(\xi\oplus\eta)=f(\xi)+f(\eta)$. 

Finally, assume that $\xi$ is orientable and $\eta$ is not. Then $\xi\oplus \eta$ is not orientable and hence 
$f(\xi\oplus \eta)=1$, $f(\xi)=0$ and $f(\eta)=1$. So in this case we have $f(\xi\oplus\eta)=f(\xi)+f(\eta)$. This completes the proof that $f$ is a semi-group homomorphism. 
 
This gives rise to a surjective homomorphism 
$$\widetilde{f}: KO(X)\longrightarrow \mathbb Z_2$$
defined by $\widetilde{f}(\xi-\eta)=f(\xi)-f(\eta)$. It is now clear that $\tilde{f}$ is zero on the $\mathbb Z$ summand of $KO(X)=\mathbb Z\oplus \widetilde{KO}(X)$ and restricts to an epimorphism $\tilde{f}:\widetilde{KO}(X)\longrightarrow \mathbb Z_2$. 
This completes the proof. \qed

\section{Computations and examples}

In this section we give a proof of Theorem\,\ref{mainproposition} and compute the characteristic rank of vector bundles over $X$, where $X$ is one of the following: the product of spheres $S^d\times S^m$, the real or complex projective space, the product space $S^1\times \mathbb C\mathbb P^n$, the Moore space $M(\mathbb Z_2,n)$ and the stunted projective space $\mathbb R\mathbb P^n/\mathbb R\mathbb P^m$. 

We begin by describing the characteristic rank of vector bundles over  $X=S^d\times S^m$.  
First note that if $d=m$, then as $r_X=d$ and $\mathrm{dim}_{\mathbb Z_2}H^d(X;\mathbb Z_2)=2$, it follows from Lemma\,\ref{secondlemma} (1) that 
$\mathrm{ucharrank}(X)=r_X-1=d-1$. 

\begin{Lem}
Let $X=S^d\times S^m$ with $d<m$. Then, 
$$\mathrm{ucharrank}(X)=\left\{ \begin{array}{cl}
                               d-1 & \mbox{if $d\neq 1, 2,4, 8$,}\\
                               m-1  & \mbox{if $d=1,2,4, 8$, $m\neq 2,4,8$}\\
                               d+m  & \mbox{if $d,m=1,2,4,8$.}
                               \end{array}\right.
                               $$
\end{Lem}
\begin{proof}  The lemma follows from the observations made after Theorem\,\ref{atiyah}. We note that $r_X=d$ and consider the maps 
$$S^d\stackrel{i}\longrightarrow S^d\times S^m\stackrel{\pi_1}\longrightarrow S^d,$$
$$S^m\stackrel{j}\longrightarrow S^d\times S^m\stackrel{\pi_2}\longrightarrow S^m,$$
where $i$ is the map $x\mapsto (x,y)$ for a fixed $y\in S^m$ and $\pi_1$ and $\pi_2$ are projections onto the the first and second factors. The map $j$ is similarly defined. 
The homomorphisms $i^*$ and $j^*$ are isomorphisms (with inverses $\pi_1^*$ and $\pi_2^*$ respectively) in degree $d$ and $m$ respectively. 

Assume that $d\neq 1,2,4,8$ and let $\xi$ be a vector bundle over $X$. Then as $w_d(i^*\xi)=0$, it follows that $w_d(\xi)=0$. Thus by Lemma\,\ref{firstlemma} (1) we have $\mathrm{charrank}_X(\xi)=r_X-1=d-1$. 

Next assume that $d=1,2,4,8$ and $m\neq 2,4,8$. Let $\nu_d$ denote the Hopf bundle over $S^d$. As $w_d(\nu_d)\neq 0$, it follows that $w_d(\pi_1^*\nu_d)\neq 0$. 
Thus $\mathrm{charrank}_{\pi_1^*\nu_d}(X)\geq m-1$. Since $m\neq 1,2,4,8$, for any vector bundle $\xi$ over $X$ we must have $w_m(\xi)=0$. This completes the proof that $\mathrm{charrank}_{\pi_1^*\nu_d}(X)=m-1$ and that $\mathrm{ucharrank}(X)=m-1$. 

Finally, let $d=1,2,4,8$ and $m=1,2,4,8$. Let $\nu_d$ and $\nu_m$ denote the Hopf bundles over $S^d$ and $S^m$ respectively. Then, clearly $w_d(\pi_1^*\nu_d\oplus \pi_2^*\nu_m)\neq 0$, $w_m(\pi_1^*\nu_d\oplus \pi_2^*\nu_m)\neq 0$ and $w_{d+m}(\pi_1^*\nu_d\oplus \pi_2^*\nu_m)\neq 0$. This shows that in this case 
$\mathrm{charrank}(X)=d+m$. This completes the proof of the lemma.\end{proof}

We now come to the proof of Theorem\,\ref{mainproposition}. First recall that the Dold manifold $P(m,n)$ is an $(m+2n)$-dimensional manifold 
defined as the quotient of $S^m\times\mathbb C\mathbb P^n$ by the fixed point free involution $(x,z)\mapsto (-x,\bar{z})$. This gives rise to a two-fold covering 
$$\mathbb Z_2\hookrightarrow S^m\times \mathbb C\mathbb P^n\longrightarrow P(m,n),$$
and via the projection $S^m\times \mathbb C\mathbb P^n\longrightarrow S^m$, a fiber bundle 
$$\mathbb C\mathbb P^n\hookrightarrow P(m,n)\longrightarrow \mathbb R\mathbb P^m$$
with fiber $\mathbb C\mathbb P^n$ and structure group $\mathbb Z_2$. In particular, for $n=1$, we have a fiber bundle 
$$S^2\hookrightarrow P(m,1)\longrightarrow \mathbb R\mathbb P^m.$$ 

The $\mathbb Z_2$-cohomology ring of $P(m,n)$ is given by \cite{do} 
$$H^*(P(m,n);\mathbb Z_2) = \mathbb Z_2[c,d]/(c^{m+1}, d^{n+1})$$
where $c\in H^1(P(m,n);\mathbb Z_2)$ and $d\in H^2(P(m,n);\mathbb Z_2)$. 

We shall make use of the following result which shows the existence of certain bundles with suitable Stiefel-Whitney classes. 
\begin{Prop}{\rm(\cite{stong}, page 86) } \label{ust} Over $P(m,n)$,
\begin{enumerate}
\item there exists a line bundle $\xi$ with total Stiefel-Whitney class $w(\xi)=1+c$;
\item there exists a $2$-plane bundle $\eta$ with total Stiefel-Whitney class $w(\eta)=1+c+d$.  
\qed 
\end{enumerate}
\end{Prop}

{\bf Proof of Theorem\,\ref{mainproposition}.}
Let $X=\mathbb R\mathbb P^n$ be the real projective space. Then $r_X=1$. Let $\xi$ be a vector bundle over $X$. If $\xi$ is orientable, then $w_1(\xi)=0$ and hence, by Lemma\,\ref{firstlemma} (1), $\mathrm{charrank}_X(\xi)=0$. On the other hand if $\xi$ is non-orientable, then $w_1(\xi)\neq 0$ and hence $\mathrm{charrank}_X(\xi)=n$ as 
$H^*(X;\mathbb Z_2)$ is polynomially generated by the non-zero element in $H^1(X;\mathbb Z_2)$. This proves (1). 

To prove (2), let $X=S^1\times \mathbb C\mathbb P^n$, then $r_X=1$. The $\mathbb Z_2$-cohomology ring of $X$ is given by 
$$H^*(X;\mathbb Z_2)=H^*(S^1;\mathbb Z_2)\otimes H^*(\mathbb C\mathbb P^n;\mathbb Z_2)\cong\mathbb Z_2[a,b]/(a^2, b^{n+1}),$$
where $a$ is of degree one and $b$ is of degree two. Let $\xi$ be a vector bundle over $X$. Evidently, $\mathrm{charrank}_X(\xi)$ is completely determined by 
the first two Stiefel-Whitney classes of $\xi$. 

We look at several cases. 
If $w_1(\xi)$ and $w_2(\xi)$ are both non-zero, then the description of the cohomology ring 
$H^*(X;\mathbb Z_2)$ forces $\mathrm{charrank}_X(\xi)=2n+1$. If $w_1(\xi)=0$, we have $\mathrm{charrank}_X(\xi)=0$. If $w_1(\xi)\neq 0$ and $w_2(\xi)=0$, then $\mathrm{charrank}_X(\xi)=1$. This completes the proof of (2). 

Finally, the proof of (3) is similar to the case (2) above in view of Proposition\,\ref{ust}. Indeed, if $w_1(\eta)=c\neq 0$ and $w_2(\eta)=d\neq 0$
(there exists such an $\eta$; see Proposition\, \ref{ust}), then we have
$\mathrm{charrank}_X(\eta)=2n+m$. If $w_1(\xi)=c\neq 0$ and $w_2(\xi)=0$ (there exists
such a $\xi$; see Proposition 4.2), we have $\mathrm{charrank}_X(\xi)=1$, as $c^2\neq
d$. For other possible vector bundles, the situation is clear. This completes the proof of (3) and the theorem.\qed

\begin{Rem} \label{importantremark} (1) We remark that, in the case (2) of the  theorem above, there exists a line bundle $\gamma$ over $X$ such that $w_1(\gamma)\neq 0$. Thus, $\mathrm{charrank}_X(\gamma)=1$. We also can  find a $2$-plane bundle $\eta$ over $X$ such that $w_1(\eta)=0$ and $w_2(\eta)\neq 0$. Thus $\mathrm{charrank}_X(\eta)=0$. Then for the Whitney sum $\gamma\oplus\eta$ we have 
$w_1(\gamma\oplus\eta)=w_1(\gamma)\neq 0$ and $w_2(\gamma\oplus\eta)=w_2(\eta)\neq 0$ and hence 
$\mathrm{charrank}_{\gamma\oplus\eta}(X)=2n+1$. The bundles $\gamma$ and $\eta$ can be obtained as the pull backs of suitable canonical bundles over $S^1=\mathbb R\mathbb P^1$ and $\mathbb C\mathbb P^n$ via the projections. 
Thus, over $X=S^1\times \mathbb C\mathbb P^n$, there  exist vector bundles having all the three possible characteristic ranks.

(2) The function $f:\mathrm{Vect}_{\mathbb R}(X)\longrightarrow \mathbb Z_2$ constructed in the proof of Theorem\,\ref{maintheorem} is in general not a semi-ring homomorphism. For example, 
let $\gamma$ denote the canonical line bundle over $X=\mathbb R\mathbb P^n$ ($n$ odd). Then $w_1(\gamma)\neq 0$ and hence $f(\gamma)=1\in\mathbb Z_2$. Now, as $\gamma\otimes \gamma$ is a trivial bundle, we have
$w_1(\gamma\otimes \gamma)=0$ and therefore, $f(\gamma\otimes\gamma)=0\in\mathbb Z_2$. Clearly, $0=f(\gamma\otimes\gamma)\neq f(\gamma)\cdot f(\gamma)=1$.
\end{Rem}

\begin{Exm}\label{complexprojective}
Let $X=\mathbb C\mathbb P^n$ be the complex projective space. Then $r_X=2$. Let $\xi$ be a vector bundle over $X$. Then 
$\mathrm{charrank}_X(\xi)=1$ if $w_2(\xi)=0$ and $\mathrm{charrank}_X(\xi)=2n$ if $w_2(X)\neq 0$. For the canonical (complex) line bundle $\gamma$ over $X$ 
we have $\mathrm{charrank}_X(\gamma)=2n$. 
\end{Exm}

We now give some examples where the bound for the cup length given by Theorem\,\ref{zeroththeorem} is sharper than that given by Theorem\,\ref{korbastheorem}. 

\begin{Exm}\label{4.6}
Let $X= S^2\times S^6$ and let $\pi_1:X\longrightarrow S^2$ be the projection. Let, as usual, $\nu_2$ denote the Hopf bundle over $S^2$. Then, $\mathrm{charrank}_{TX}(X)=1$, and $\mathrm{charrank}_X(\xi)=5$ where $\xi=\pi_1^*\nu_2$. The bundle $\xi$ satisfies the condition of Theorem\,\ref{zeroththeorem} with $k=5$. 
 Then the bound for the cup length, $\mathrm{cup}(X$), of $X$ given by Theorem\,\ref{korbastheorem} is $4$ and that given by Theorem\,\ref{zeroththeorem} is $2$. 
\end{Exm}

\begin{Exm}\label{4.7} 
Let $X = S^4\times S^8$. Let $\xi=\pi_1^*\nu_4\oplus \pi_2^*\nu_8$. Then, $\mathrm{charrank}_{TX}(X)=3$ and $\mathrm{charrank}_X(\xi)= 12$. Then $\xi$ satisfies the condition of Theorem\,\ref{zeroththeorem} with $k=7$. Then the bound for the cup length, $\mathrm{cup}(X$), of $X$ given by Theorem\,\ref{korbastheorem} is $3$ and that given by Theorem\,\ref{zeroththeorem} is $2$.
\end{Exm}

\begin{Rem}
These sharper estimates of Examples \ref{4.6} and \ref{4.7} can also be obtained from Theorem A \cite{korbas2}.
\end{Rem}

We now compute $\mathrm{charrank}_X(\xi)$ where $X$ is the Moore space $M(\mathbb Z_2,n)$, $n>1$, and $\xi$ a vector bundle over $X$. 
We recall that $X$ is an $(n-1)$-connected $(n+1)$-dimensional $CW$-complex. Note that $M(\mathbb Z_2, 1)$ is the real projective space $\mathbb R\mathbb P^2$ and $M(\mathbb Z_2,n)$ is the iterated suspension $\Sigma^nM(\mathbb Z_2,1)$. 
We refer to \cite{hatcher} for basic properties of Moore spaces. 
We prove the following.

\begin{Prop} \label{secondlastprop}
Let $X$ denote the Moore space $M(\mathbb Z_2,n)$ with $n>1$. Then,  
$$\mathrm{ucharrank}(X)= \left\{\begin{array}{cl}
 n-1 & \mbox{if $n\neq 2$}\\
 3 & \mbox{if $n=2$}
\end{array}\right.$$
\end{Prop}
\begin{proof} The Moore space $X$ is an $(n+1)$-dimensional $CW$-complex with $n$-skeleton $S^n$. Let $i:S^n\hookrightarrow X$ denote the inclusion map. 
Using the cellular chain complex, for example, it is easy to see that the homomorphism 
$$i^*:H^n(X;\mathbb Z_2)\longrightarrow H^n(S^n;\mathbb Z_2)$$
in degree $n$ is an isomorphism and hence the non-zero element in $H^n(X;\mathbb Z_2)$ is spherical. 

Assume that $n\neq 2,4,8$. Since $X$ is $(n-1)$-connected it follows from Proposition\,\ref{covering} that $\mathrm{charrank}_X(\xi)=n-1$ for any $\xi$ over $X$. This proves the first equality for $n\neq 2, 4,8$.  

Next, for $X=M(\mathbb Z_2,n)$, we observe that there is a cofiber sequence 
$$S^n\stackrel{f}\longrightarrow S^n \longrightarrow X\longrightarrow S^{n+1}\longrightarrow S^{n+1}$$
where $f$ is a degree $2$ map. This gives rise to an exact sequence 
$$\widetilde{KO}(S^{n+1})\longrightarrow \widetilde{KO}(X)\longrightarrow \widetilde{KO}(S^n)\stackrel{f^*}\longrightarrow \widetilde{KO}(S^n).$$
When $n=4,8$ the homomorphism $f^*$ is injective and hence the homomorphism $\widetilde{KO}(S^{n+1})\longrightarrow \widetilde{KO}(X)$ is surjective. 
When $n=2$, the homomorphism $f^*$ is the zero homomorphism and hence the homomorphism $\widetilde{KO}(X)\longrightarrow \widetilde{KO}(S^n)$ is surjective. 
These obeservations follow from the fact that $\widetilde{KO}(S^4)=\mathbb Z=\widetilde{KO}(S^8)$ and $\widetilde{KO}(S^2)=\mathbb Z_2$ together with the fact that $f$ is a degree $2$ map. 

Thus when $n=4,8$ we have by Theorem\,\ref{atiyah} that $w(\xi)=1$ for any vector bundle over $X=M(\mathbb Z_2,n)$. This completes the proof of the first equality when $n=4,8$. 

Finally let $X=M(\mathbb Z_2,2)$. Then $X$ is a simply connected $3$-dimensional $CW$-complex. We shall show that there exists 
a bundle $\xi$ over $X$ with $w_2(\xi)\neq 0$ and $w_3(\xi)\neq 0$. As the homomorphism $\widetilde{KO}(X)\longrightarrow \widetilde{KO}(S^2)$ is surjective and $w_2(\nu_2)\neq 0$,  
there exists a bundle $\xi$ over $X$ with $w_2(\xi)\neq 0$. For this vector bundle $\xi$ over $X$ the Stiefel-Whitney class $w_3(\xi)\neq 0$. To see this we observe that if $a\in H^1(\mathbb R\mathbb P^2;\mathbb Z_2)= H^1(M(\mathbb Z_2,1);\mathbb Z_2)$ is the unique non-zero element, then
$Sq^1(a)=a^2\neq 0$. Thus, by Wu's formula and the fact that the Steenrod squares commute with the suspension homomorphism we see that 
$Sq^1(w_2(\xi))=w_1(\xi)w_2(\xi)+w_3(\xi)=w_3(\xi)\neq 0$. This completes the proof of the second equality.\end{proof}

\begin{Prop}\label{lastprop}
Let $X$ denote the stunted projective space $\mathbb R\mathbb P^n/\mathbb R\mathbb P^m$ with $1\leq m\leq n-2$. Then 
$$\mathrm{ucharrank}(X)=\left\{\begin{array}{cl}
 m & \mbox{if $m+1\neq 2,4,8$}\\
 m+1 & \mbox{if $m+1=2,4,8$}
\end{array}\right.$$

\end{Prop}
\begin{proof} The stunted projective space $X$ is $m$-connected with $(m+1)$-skeleton $X^{(m+1)}=S^{m+1}$. If $f:S^{m+1}=X^{(m+1)}\longrightarrow X$ denotes the inclusion map, then it is easy to check that the homomorphism 
$$f^*:H^{m+1}(X;\mathbb Z_2)\longrightarrow H^{m+1}(S^{m+1};\mathbb Z_2)$$ is an isomorphism. Thus, the non-zero element in $H^{m+1}(X;\mathbb Z_2)$ is spherical. 
The first equality of the proposition now follows from Proposition\,\ref{covering}. 

Let $X=\mathbb R\mathbb P^n/\mathbb R\mathbb P^m$ with $m+1=2,4,8$. It is clear that the inclusion map 
$$\mathbb R\mathbb P^{m+2}/\mathbb R\mathbb P^m \longrightarrow \mathbb R\mathbb P^n/\mathbb R\mathbb P^m$$
where $n\geq m+2$ induces isomorphism in $\mathbb Z_2$-cohomology in degree $i$ for all $i\leq m+2$. Since $(m+2)$ is odd we have a splitting 
$$\mathbb R\mathbb P^{m+2}/\mathbb R\mathbb P^m = S^{m+2}\vee S^{m+1}.$$
It follows that $X$ has a spherical class in degree $(m+2)$ and hence by Proposition\,\ref{covering} we have $\mathrm{ucharrank}(X)\leq m+1$. 
We shall prove the equality by showing that there exists a bundle $\xi$ over $X$ with $w_{m+1}(\xi)\neq 0$. 

As $\mathbb R\mathbb P^{m+2}/\mathbb R\mathbb P^m=S^{m+1}\vee S^{m+2}$, the Hopf bundle $\nu_{m+1}$ over $S^{m+1}$ extends over $S^{m+1}\vee S^{m+2}$ to give a vector bundle $\xi$ with $w_{m+1}(\xi)\neq 0$. It is well known \cite{adams} that for any $n\geq m+2$ the inclusion map 
$$\mathbb R\mathbb P^{m+2}/\mathbb R\mathbb P^m\hookrightarrow \mathbb R\mathbb P^n/\mathbb R \mathbb P^m$$
induces an epimorphism in reduced $KO$-groups. Thus there is a vector bundle over $\mathbb R\mathbb P^n/\mathbb R \mathbb P^m$ with the required property. \end{proof}

\begin{acknowledgement} We are indebted to Professor J. Korba\v{s} for his detailed and helpful comments on an earlier draft of this manuscript. In particular, we thank him for showing us the proof of Corollary\,\ref{lens1}. The original statement of the corollary only contained the conclusion that $\mathrm{ucharrank}(X)<d$, under the assumption that $X$ is orientable and $d\neq 1,2,4,8$. We also thank him for sending us a copy of his paper \cite{korbas}. We would like to thank the anonymous referee for his detailed suggestions. In particular, we thank him for showing us the proof of Proposition\,\ref{secondlastprop}. This is shorter and stronger than proof given by the authors. 
\end{acknowledgement}

\end{document}